\documentclass[12pt]{amsart}

\usepackage[T1]{fontenc}
\usepackage[utf8]{inputenc}
\usepackage{hyperref}
\usepackage{libertinus}
\usepackage{amsmath,amssymb,amsthm,mathtools}

\usepackage{geometry}
\usepackage{enumitem}
\usepackage{microtype}

\newtheorem{theorem}{Theorem}
\newtheorem{corollary}[theorem]{Corollary}
\newtheorem{proposition}[theorem]{Proposition}
\newtheorem{lemma}[theorem]{Lemma}
\newtheorem{remark}[theorem]{Remark}

\newcommand{\R}{\mathbb R}
\newcommand{\C}{\mathbb C}
\newcommand{\N}{\mathbb N}
\newcommand{\Q}{\mathbb Q}
\newcommand{\inner}[2]{\left\langle #1,#2\right\rangle}
\newcommand{\norm}[1]{\left\|#1\right\|}
\newcommand{\diag}{\operatorname{diag}}

\newcommand{\abs}[1]{\left|#1\right|}

\title[Strong Polarization and Entropy]
{Strong Polarization and Entropy}

\author[D. Galicer]{Daniel Galicer}
\address{
	Departamento de Matem\'atica y Estad\'istica, Universidad Torcuato Di Tella,
	Av. Figueroa Alcorta 7350 (1428),
	Buenos Aires, Argentina and IMAS-CONICET. 
}
\email{daniel.galicer@utdt.edu}

\author[O. Ortega--Moreno]{Oscar Ortega--Moreno}
\address{
	Departamento de Matem\'atica, CUNEF Universidad,
	Madrid, Spain
}
\email{oscar.ortegamoreno@cunef.edu}

\author[D. Pinasco]{Dami\'an Pinasco}
\address{
	Departamento de Matem\'atica y Estad\'istica, Universidad Torcuato Di Tella, Av. Figueroa Alcorta 7350 (1428),
	Buenos Aires, Argentina and CONICET.
}
\email{dpinasco@utdt.edu}

\date{}

\begin{document}
	
	\maketitle
	
	\begin{abstract}
		We show that for any set of $n$ unit vectors $v_1,\ldots,v_n$ in a real Hilbert space and positive numbers $p_1,\ldots,p_n$ satisfying $\sum_jp_j = 1$, there exists a unit vector $u$ such that
		\[
		\sum_{j=1}^n \frac{p_j^2}{\langle v_j, u\rangle^2}\leq 1.
		\]
		This inequality is a weighted version of the strong polarization inequality. As immediate corollaries, it yields a polarization inequality for products of powers of linear functionals and a strengthening of Bang’s classical plank theorem for Hilbert spaces. The proof follows the approach introduced by Martínez and Ortega-Moreno in their recent solution to the strong polarization conjecture posed by Ball and Frenkel. We further note that our weighted inequality admits a Shannon-entropy interpretation: in a random sensing model, the entropy of the weights controls the minimum expected logarithmic loss.
		
	\end{abstract}
	\section{Introduction}
	
	In this paper we prove the following generalization of the strong polarization inequality, conjectured by Ball and Frenkel, an recently solved by Martínez and Ortega-Moreno in \cite{MartinezOrtega}.
	
	\begin{theorem}[Weighted strong polarization inequality]\label{thm:strong}
		Let \(H\) be a real Hilbert space, let \(v_1,\ldots,v_n\in H\) be unit vectors, and let
		\(p_1,\ldots,p_n>0\) be weights satisfying
		\[
		\sum_{j=1}^n p_j = 1.
		\]
		Then there exists a unit vector \(u\) such that
		\begin{equation}\label{inq:strong}
			\sum_{j=1}^n \frac{p_j^2}{\langle v_j,u\rangle^2} \leq 1.
		\end{equation}
		
	\end{theorem}
	
	The original conjecture corresponds to the case when all the weights are equal; that is, when $p_j = \frac{1}{n}$ for all $j$. This weighted version of the strong polarization inequality has a few immediate consequences. One of them is obtain by applying the weighted arithmetic-geometric mean inequality, yielding the following weighted version of the polarization inequality for products of powers of linear functional.

	\begin{theorem}[Weighted polarization inequality]\label{thm:weak}
	Under the assumptions of Theorem \ref{thm:strong},
	\[
	\sup_{\norm{u}=1}
	\prod_{j=1}^n \abs{\inner{v_j}{u}}^{p_i}
	\ge
	\prod_{j=1}^n p_i^{p_i/2}.
	\]
	\end{theorem}
	
	The case when all weights are equals recovers the polarization inequality as well.\\

	The problem of estimating the norm of a product of linear functionals, or more generally the norm of a product of homogeneous polynomials, has been studied in several settings. The polarization problem asks how small the product of \(n\) norm-one linear functionals can be on the unit ball of a Banach or Hilbert space. It led, for a Banach space $X$, to the study of its polarization constant \(c_n(X)\) defined as the smallest constant \(C\) such that \[ C \norm{f_1 \cdots f_n} \geq \|f_1\|\cdots \|f_n\| \] for every choice of functionals \(f_1,\ldots,f_n\in X^*\). An immediate consequence of Ball's plank theorem is the general Banach-space bound \(c_n(X)\leq n^n\), while the sharper conjecture for real Hilbert spaces, namely \(c_n(H)\le n^{n/2}\), remained open for decades. Several works contributed to the development of this conjecture  \cite{BST,Frenkel,LeungLiRakesh,MatolcsiMunoz,PappasRevesz,PinascoRn,ReveszSarantopoulos}. The complex case was settled by Arias de Reyna \cite{Arias} and Ball \cite{BallComplex} and the equality case was characterized by Pinasco \cite{PinascoBombieri}, while recent work of Martínez and Ortega-Moreno resolves the real case through a variational approach and a elegant application of the Euler--Jacobi vanishing theorem \cite{MartinezOrtega}.\\

	Theorem~\ref{thm:strong} can also be seen as a strengthening of the classic polarization theorem for Hilbert spaces. Indeed, since each summand in \eqref{inq:strong} is less or equal to one, we obtain
	\[
	\abs{\inner{v_j}{u}}\geq p_j,
	\]
	for all $j=1,\ldots,n$.
	
	\subsection*{Shannon Entropy and sensing}
	
	The weighted polarization inequalities admit a natural interpretation in
	terms of sensing systems and information distribution. The vectors \(v_j\) may be viewed as sensing directions, frame elements, or measurement vectors, while $\abs{\langle v_j,u\rangle}^2$ represents the amount of energy captured from a signal \(u\in S_H\) (here  $S_H$ denotes the set of unit vector in $H$). The weighted polarization inequality ensures the existence of \(u\in S_H\)
	such that
	\[
	-\sum_{j=1}^n
	p_j \log\abs{\langle v_j,u\rangle}^2
	\le
	-\sum_{j=1}^n p_j\log p_j.
	\]
	The quantity on the right-hand side is the Shannon entropy
	\[
	H(p)=-\sum_{j=1}^n p_j\log p_j
	\]
	of the weight distribution. Now let \(V\) be a random sensing vector with distribution
	\[
	\mathbb P(V=v_j)=p_j,
	\]
	for all $j=1,\ldots,n$. For \(u\in S_H\), define
	$$
	X_u=-\log |\langle V,u\rangle|^2.
	$$
	Then
	\[
	\mathbb E X_u
	= - \sum_{j=1}^n p_j \log |\langle v_j,u\rangle|^2.
	\]
	Therefore, the theorem guarantees the existence of \(u\in S_H\) such that
	\[
	\mathbb E X_u \le H(p).
	\]
	In other words, there is always a signal whose expected logarithmic loss under random sensing is controlled by the entropy of the sensing distribution. This viewpoint is natural in frame theory and harmonic analysis. The weights \(p_j\) describe the distribution of the sensing architecture, while the entropy \(H(p)\) measures its effective complexity. The theorem says that this complexity controls the best possible average logarithmic correlation with the system. On the other hand, the stronger estimate
	\[
	\sum_{j=1}^n
	\frac{p_j^2}{|\langle v_j,u\rangle|^2}
	\le 1
	\]
	adds a reciprocal energy constraint: the correlations with the sensing system cannot all be too small relative to the prescribed distribution.\\

	The article is organized as follows. In Section~\ref{proofmain}, we prove Theorem~\ref{thm:strong} and we discuss its sharpness. In Section~\ref{sec:con1}, we prove the polarization inequality and its implication for the plank problem in Hilbert spaces and the polarization constant of $\ell_p$ spaces.

	\section{Proof of Theorem~\ref{thm:strong}}\label{proofmain}
	
	The proof is divided in two steps: first we show that Theorem~\ref{thm:strong} is true in the case of rational weights. Second we obtain the general result by a simple approximation argument. Our approach follows closely the proof of the main theorem in \cite{MartinezOrtega}: a careful analysis of the extremal point method and the use of the Euler--Jacobi vanishing theorem.
	
	\subsection{Integer multiplicities}
	
	We first prove Theorem~\ref{thm:strong} for probability vectors of the form \(\alpha_j=k_j/s\), where \(k_1,\ldots,k_n\in\N\) and \(s=k_1+\cdots+k_n\). Thus the integers \(k_j\) should be regarded as multiplicities, or equivalently as the powers appearing in the auxiliary polynomial below.
	
	Let \(k_1,\ldots,k_n\in\N\), put
	\[
	s=k_1+\cdots+k_n,
	\qquad
	\alpha_j=\frac{k_j}{s},
	\]
	and suppose first that \(v_1,\ldots,v_n\) is a basis of \(\R^n\). We shall prove that there exists \(u\in S^{n-1}\) such that
	\[
	\sum_{j=1}^n\frac{k_j^2}{\inner{v_j}{u}^2}\le s^2.
	\]
	Equivalently,
	\[
	\sum_{j=1}^n\frac{\alpha_j^2}{\inner{v_j}{u}^2}\le1.
	\]
	
	Let \(w_1,\ldots,w_n\) be the dual basis, so that \(\inner{v_j}{w_i}=\delta_{ji}\). For \(x\in\R^n\), set
	\[
	y_j=\inner{v_j}{x},
	\qquad
	z_j=\inner{w_j}{x}.
	\]
	Then \(x=\sum_j z_j v_j\). Let \(G=(\inner{v_j}{v_k})_{j,k}\) be the Gram matrix and let \(A=G^{-1}\). In vector notation, \(z=Ay\).
	
	Consider
	\[
	P(x)=\prod_{j=1}^n \inner{v_j}{x}^{k_j}.
	\]
	At a critical point of \(P\) on the sphere where \(P\ne0\), the Lagrange multiplier equation for \(\log |P|\) is
	\[
	\sum_{j=1}^n k_j\frac{v_j}{y_j}=\lambda x.
	\]
	Taking the inner product with \(x\), we get \(\lambda=s\). Evaluating against \(w_j\), we obtain
	\[
	\frac{k_j}{y_j}=sz_j,
	\]
	or equivalently
	\[
	y_jz_j=\frac{k_j}{s}=\alpha_j.
	\]
	Since \(z=Ay\), the extremal points are described by the quadratic system
	\[
	y_j(Ay)_j=\alpha_j,
	\qquad j=1,\ldots,n.
	\]
	Define \(h=(h_1,\ldots,h_n):\C^n\to\C^n\) by
	\[
	h_j(y)=y_j(Ay)_j-\alpha_j.
	\]
	
	\begin{lemma}\label{lem:complex-real}
		Every solution \(y\in\C^n\) of \(h(y)=0\) is real.
	\end{lemma}
	
	\begin{proof}
		Let \(y\in\C^n\) be a solution. Since \(\alpha_j>0\), no \(y_j\) vanishes. Put \(z=Ay\) and
		\[
		x=\sum_{j=1}^n z_j v_j\in\C^n.
		\]
		Then \(y_j=\inner{v_j}{x}\). Moreover, from \(y_jz_j=k_j/s\),
		\[
		z_j=\frac{k_j}{s\,y_j},
		\]
		and therefore
		\[
		x=\frac1s\sum_{j=1}^n k_j\frac{v_j}{\inner{v_j}{x}}.
		\]
		Write \(x=a+ib\), with \(a,b\in\R^n\). Since the vectors \(v_j\) are real,
		\[
		\inner{v_j}{x}=\inner{v_j}{a}+i\inner{v_j}{b}.
		\]
		Set \(A_j=\inner{v_j}{a}\) and \(B_j=\inner{v_j}{b}\). Taking imaginary parts in the fixed point equation gives
		\[
		sb=-\sum_{j=1}^n k_j\frac{B_j}{A_j^2+B_j^2}v_j.
		\]
		Taking the real inner product with \(b\),
		\[
		s\norm{b}^2=-\sum_{j=1}^n k_j\frac{B_j^2}{A_j^2+B_j^2}.
		\]
		The left-hand side is non-negative and the right-hand side is non-positive. Hence \(b=0\), and so \(x\), and therefore \(y\), is real.
	\end{proof}
	
	We will use the following lemma.
	
	\begin{lemma}[\cite{MartinezOrtega}]\label{lem:chambers}
		In each chamber determined by the signs of \(y_1,\ldots,y_n\), the system \(h(y)=0\) has a unique real solution.
	\end{lemma}
	
	%\begin{proof}
	%	Consider
	%	\[
	%	\Psi(x)=\frac12\norm{x}^2-\frac1s\sum_{i=1}^n k_i\log |\inner{v_i}{x}|.
	%	\]
	%	In each chamber, \(\Psi\) is strictly convex because
	%	\[
	%	\nabla^2\Psi(x)
	%	=
	%	I+
	%	\frac1s\sum_{i=1}^n k_i\frac{v_i\otimes v_i}{\inner{v_i}{x}^2}
	%	\]
	%	is positive definite. Moreover, \(\Psi(x)\to+\infty\) when \(x\) approaches the boundary of the chamber, and also as \(\norm{x}\to\infty\). Thus \(\Psi\) has a unique minimum in each chamber. Its Euler equation is
	%	\[
	%	x=\frac1s\sum_{i=1}^n k_i\frac{v_i}{\inner{v_i}{x}},
	%	\]
	%	which is equivalent to \(y_i(Ay)_i=\alpha_i\).
	%\end{proof}
	
	It follows from Lemmas~\ref{lem:complex-real} and \ref{lem:chambers} that
	\(h(y)=0\) has exactly \(2^n\) solutions in \(\C^n\), all of them real. Moreover, all solutions are simple. Indeed, at a solution,
	\[
	Jh(y)=\diag(y)
	\left(A+\diag\left(\frac{\alpha_j}{y_j^2}\right)\right),
	\]
	and this matrix is invertible since \(A\) is positive definite and \(y_j\ne0\).
	
	We now apply the Euler--Jacobi vanishing theorem. Since the system has \(2^n\)
	simple solutions and each \(h_j\) has degree \(2\), every polynomial \(g\) with
	\(\deg g\le n-1\) satisfies
	\[
	\sum_{h(y)=0}\frac{g(y)}{\det Jh(y)}=0.
	\]
	Assume first that \(n\ge2\), the case \(n=1\) being immediate. Put
	\[
	Q(y)=\prod_{j=1}^n y_j
	\]
	and define
	\[
	g(y)=2\sum_{1\le i<j\le n}k_i k_j\inner{v_i}{v_j}
	\prod_{\ell\ne i,j}y_\ell.
	\]
	Then \(\deg g=n-2\). Also,
	\[
	\frac{g(y)}{Q(y)}
	=2\sum_{1\le i<j\le n}k_i k_j\frac{\inner{v_i}{v_j}}{y_iy_j}.
	\]
	On the other hand,
	\[
	\left\|\sum_{j=1}^n k_j\frac{v_j}{y_j}\right\|^2
	=
	\sum_{j=1}^n\frac{k_j^2}{y_j^2}
	+2\sum_{1\le i<j\le n}k_i k_j\frac{\inner{v_i}{v_j}}{y_iy_j}.
	\]
	At a solution of \(h(y)=0\), the associated vector \(x\) satisfies
	\[
	\sum_{j=1}^n k_j\frac{v_j}{y_j}=sx
	\]
	and
	\[
	\norm{x}^2=\sum_{i=1}^n y_i z_i=\sum_{i=1}^n\alpha_i=1.
	\]
	Thus, at every solution,
	\[
	g(y)=Q(y)\left(s^2-\sum_{j=1}^n\frac{k_j^2}{y_j^2}\right).
	\]
	Moreover,
	\[
	\det Jh(y)
	=Q(y)
	\det\left(A+\diag\left(\frac{\alpha_j}{y_j^2}\right)\right).
	\]
	Set
	\[
	\mu(y)=
	\det\left(A+\diag\left(\frac{\alpha_j}{y_j^2}\right)\right)^{-1}.
	\]
	Then \(\mu(y)>0\) for every solution. Euler--Jacobi gives
	\[
	\sum_{h(y)=0}
	\left(s^2-\sum_{j=1}^n\frac{k_j^2}{y_j^2}\right)\mu(y)=0.
	\]
	Since all \(\mu(y)\) are positive, there exists a solution \(y\) such that
	\[
	\sum_{j=1}^n\frac{k_j^2}{y_j^2}\le s^2.
	\]
	For the corresponding unit vector \(u\), this is precisely the case
	\[
	\sum_{j=1}^n\frac{k_j^2}{\inner{v_j}{u}^2}\le s^2.
	\]
	This proves Theorem~\ref{thm:strong} for weights of the form
	\(\alpha_j=k_j/s\) when \(v_1,\ldots,v_n\) is a basis of \(\R^n\).
	
	We pass to arbitrary unit vectors in a real Hilbert space by the usual perturbation argument. The statement depends only on the finite-dimensional subspace generated by the vectors, so we may assume this is the ambient space. If the vectors are linearly independent, the preceding argument applies. If they are linearly dependent, embed the space in \(\R^n\) and choose unit vectors \(v_j^{(t)}\to v_j\) such that \(v_j^{(t)},\ldots,v_n^{(t)}\) is a basis of \(\R^n\) for every \(t>0\). Applying the basis case and passing to a convergent subsequence gives the desired inequality for the original vectors. If the limiting vector has a component orthogonal to the original span, projecting onto that span and normalizing can only increase the quantities \(\abs{\inner{v_i}{u}}\). This proves Theorem~\ref{thm:strong} for weights of the form
	\(\alpha_j=k_j/s\) in full generality.
	
	\subsection{Rational and real weights}
	
	The passage to rational weights is immediate. Let
	\(p_1,\ldots,p_n\in\Q_{>0}\) and suppose that \(\sum_i p_j=1\). Choose
	\(N\in\N\) so that \(k_j=Np_j\in\N\) for every \(i\). Then
	\[
	N=k_1+\cdots+k_n,
	\qquad
	p_j=\frac{k_i}{N}.
	\]
	By the multiplicity case proved above, there exists a unit vector \(u\) such
	that
	\[
	\sum_{j=1}^n\frac{k_j^2}{\inner{v_j}{u}^2}\le N^2.
	\]
	Dividing by \(N^2\),
	\[
	\sum_{j=1}^n\frac{p_j^2}{\inner{v_j}{u}^2}\le1.
	\]
	
	For real positive weights, take rational vectors \(p^{(m)}=(p_1^{(m)},\ldots,p_n^{(m)})\) with positive coordinates, \(\sum_j p_j^{(m)}=1\), and \(p_j^{(m)}\to p_j\). By the rational case, there exists \(u_m\in S_H\) such that
	\[
	\sum_{j=1}^n\frac{(p_i^{(m)})^2}{\inner{v_j}{u_m}^2}\le1.
	\]
	Restricting to the finite-dimensional span and passing to a subsequence, \(u_m\to u\in S_H\). The preceding inequality forces \(|\inner{v_j}{u_m}|\ge p_j^{(m)}\) for every \(j\), and hence \(|\inner{v_j}{u}|\ge p_j>0\). Passing to the limit gives the following.
	\[
	\sum_{j=1}^n\frac{p_j^2}{\inner{v_j}{u}^2}\le 1.
	\]
	This proves Theorem \ref{thm:strong}.
	
	\subsection{Sharpness and equality}\label{sec:sharpness}
	
	The constant \(1\) in Theorem~\ref{thm:strong} is the best possible. Indeed, if
	\(v_1,\ldots,v_n\) is an orthonormal system and \(u\in S_H\), then, writing
	\(a_j=\inner{v_j}{u}\), Cauchy's inequality gives
	\[
	1=\left(\sum_{j=1}^n p_j\right)^2
	\le
	\left(\sum_{j=1}^n\frac{p_j^2}{a_j^2}\right)
	\left(\sum_{j=1}^n a_j^2\right)
	\le
	\sum_{j=1}^n\frac{p_j^2}{a_j^2}.
	\]
	Thus, the right-hand side in Theorem~\ref{thm:strong} cannot be replaced by a
	smaller constant. Equality is attained, in the orthonormal case, in
	\[
	u=\sum_{j=1}^n \varepsilon_j\sqrt{p_j}\,v_j,
	\qquad \varepsilon_i\in\{-1,1\}.
	\]
	Consequently, the constant in the weighted polarization inequality is also the best possible, since for this vector one has
	\[
	\prod_{i=1}^n |\inner{v_i}{u}|^{p_i}
	=
	\prod_{i=1}^n p_i^{p_i/2}.
	\]
	
	\begin{remark}
    The reverse implication can be proved in a similar way to the argument of
    Martínez and Ortega Moreno in \cite{MartinezOrtega}.
    \end{remark}
	
	\section{Consequences of Theorem~\ref{thm:strong}}\label{sec:con1}
	
	The following consequence is the weighted product inequality. It is the form in which the result is closest to the usual polarization problem.

	\begin{proof}[Proof of theorem~\ref{thm:weak}]
		Let \(u\) be as in Theorem \ref{thm:strong}. Put
		\[
		\rho_j=\frac{p_j}{\inner{v_j}{u}^2}.
		\]
		Then
		\[
		\sum_{j=1}^n p_j\rho_j
		=
		\sum_{j=1}^n\frac{p_j^2}{\inner{v_j}{u}^2}
		\le1.
		\]
		By the weighted arithmetic-geometric mean inequality,
		\[
		\prod_{j=1}^n\rho_j^{p_j}\le 1.
		\]
		Equivalently,
		\[
		\prod_{j=1}^n\abs{\inner{v_j}{u}}^{p_j}
		\ge
		\prod_{j=1}^n p_j^{p_j/2}.
		\]
	\end{proof}
		
	The strong form also yields a centred plank theorem in Hilbert spaces. Since relatively few proofs of plank-type results are known, we believe that this observation is worth recording.
	
	\begin{corollary}[Plank theorem for Hilbert spaces]\label{cor:plank}
		Let \(v_1,\ldots,v_n\) be unit vectors in a real Hilbert space and let \(p_1,\ldots,p_n>0\) satisfy \(\sum_i p_i=1\). Then there exists \(u\in S_H\) such that
		\[
		\abs{\inner{v_j}{u}}\ge p_j,
		\qquad j=1,\ldots,n.
		\]
		Equivalently, the spherical planks
		\[
		\{u\in S_H: \abs{\inner{v_j}{u}}<p_j\},
		\qquad j=1,\ldots,n,
		\]
		do not cover the unit sphere.
	\end{corollary}
	
	\begin{proof}
		By Theorem \ref{thm:strong},
		\[
		\sum_{j=1}^n\frac{p_j^2}{\inner{v_j}{u}^2}\le1.
		\]
		Each summand is non-negative, hence
		\[
		\frac{p_j^2}{\inner{v_j}{u}^2}\le1
		\]
		for every \(j\). This gives \(\abs{\inner{v_j}{u}}\ge p_j\).
	\end{proof}
	
	%	\section{A consequence for \texorpdfstring{\(\ell_p^d\)}{ell_p^d}, \texorpdfstring{\(1\le p\le2\)}{1 <= p <= 2}}\label{sec:lp}
		
		The Hilbertian result also gives a simple estimate for finite-dimensional \(\ell_p\) spaces in the range \(1\le p\le2\). Let \(d\) denote the dimension and let \(n\) be the number of functionals.
		
		\begin{proposition}\label{prop:lp}
				Let \(1\le p\le2\), let \(q\) be the conjugate exponent, and let
				\[
				x_1^*,\ldots,x_n^*\in(\ell_p^d)^*=\ell_q^d
				\]
				be norm-one functionals. If \(\lambda_1,\ldots,\lambda_n>0\) and \(\Lambda=\lambda_1+\cdots+\lambda_n\), then
				\[
				\sup_{\norm{x}_p\le1}
				\prod_{j=1}^n |x_j^*(x)|^{\lambda_j}
				\ge
				d^{-\left(\frac1p-\frac12\right)\Lambda}
				\prod_{j=1}^n
				\left(\frac{\lambda_j}{\Lambda}\right)^{\lambda_j/2}.
				\]
			\end{proposition}
		
		\begin{proof}
				Write \(x_j^*(x)=\inner{a_j}{x}\), with \(\norm{a_j}_q=1\). Since \(q\ge2\), \(\norm{a_j}_2\ge\norm{a_j}_q=1\). Put
				\[
				u_j=\frac{a_j}{\norm{a_j}_2}\in S_{\ell_2^d}.
				\]
				By applying Theorem~\ref{thm:weak} with $p_j=\frac{\lambda_j}{\Lambda}$, there exists \(y\in S_{\ell_2^d}\) such that
				\[
				\prod_{j=1}^n |\inner{u_j}{y}|^{\lambda_j}
				\ge
				\prod_{j=1}^n
				\left(\frac{\lambda_j}{\Lambda}\right)^{\lambda_j/2}.
				\]
				Since \(\norm{a_j}_2\ge1\), the same lower bound holds with \(a_j\) in place of \(u_j\). Moreover, for \(1\le p\le2\),
				\[
				\norm{y}_p\le d^{1/p-1/2}\norm{y}_2=d^{1/p-1/2}.
				\]
				Taking \(x=y/\norm{y}_p\), we get \(\norm{x}_p\le1\) and
				\[
				|x_j^*(x)|
				\ge d^{-(1/p-1/2)}|\inner{a_j}{y}|.
				\]
				Multiplying with weights \(\lambda_j\) gives the result.
			\end{proof}
		
		For \(n=d\), \(x_j^*=e_j^*\), and equal weights, the estimate gives \(d^{-d/p}\), which is the exact value of
		\[
		\sup_{\norm{x}_p\le1}|x_1\cdots x_d|.
		\]
		Thus, the order in the dimension is optimal.


\begin{thebibliography}{99}
		
		\bibitem{Arias}
		J. Arias-de-Reyna,
		\emph{Gaussian variables, polynomials and permanents},
		Linear Algebra Appl. 285 (1998), 107--114.
		
		\bibitem{BallComplex}
		K. M. Ball,
		\emph{The complex plank problem},
		Bull. London Math. Soc. 33 (2001), 433--442.
		
		\bibitem{BallPlank}
		K. M. Ball,
		\emph{The plank problem for symmetric bodies},
		Invent. Math. 104 (1991), 535--543.
		
		\bibitem{Bang}
		T. Bang,
		\emph{A solution of the plank problem},
		Proc. Amer. Math. Soc. 2 (1951), 990--993.
		
		\bibitem{BST}
		C. Benítez, Y. Sarantopoulos and A. M. Tonge,
		\emph{Lower bounds for norms of products of polynomials},
		Math. Proc. Cambridge Philos. Soc. 124 (1998), 395--408.
		
		\bibitem{CarandoPinascoRodriguez}
		D. Carando, D. Pinasco and J. T. Rodríguez,
		\emph{Lower bounds for norms of products of polynomials on \(L_p\) spaces},
		Studia Math. 214 (2013), 157--166.
		
		
		\bibitem{NonlinearPlanks}
		D. Carando, D. Pinasco and J. T. Rodríguez,
		\emph{Non-linear plank problems and polynomial inequalities},
		Rev. Mat. Complut. 30 (2017), no. 3, 507--523.
		
		\bibitem{FiniteDimensional}
		D. Carando, D. Pinasco and J. T. Rodríguez,
		\emph{On the linear polarization constants of finite dimensional spaces},
		Math. Nachr. 290 (2017), 2547--2559.
		
		\bibitem{Frenkel}
		P. E. Frenkel,
		\emph{Pfaffians, hafnians and products of real linear functionals},
		Math. Res. Lett. 15 (2008), no. 2, 351--358.
		
		\bibitem{GarciaVilla}
		J. C. García-Vázquez and R. Villa,
		\emph{Lower bounds for multilinear forms defined on Hilbert spaces},
		Mathematika 46 (1999), 315--322.
		
		\bibitem{LeungLiRakesh}
		Y. J. Leung, W. V. Li and Rakesh,
		\emph{The \(d\)-th linear polarization constant of \(\R^d\)},
		J. Funct. Anal. 255 (2008), 2861--2871.
		
		\bibitem{MartinezOrtega}
		A. D. Martínez and O. Ortega-Moreno,
		\emph{A solution to the polarization problem},
		preprint, 2026.
		
		\bibitem{MatolcsiMunoz}
		M. Matolcsi and G. A. Muñoz,
		\emph{On the real linear polarization constant problem},
		Math. Inequal. Appl. 9 (2006), 485--494.
		
		\bibitem{MoslehianMunozPeraltaSeoane}
		M. S. Moslehian, G. A. Muñoz-Fernández, A. M. Peralta and J. B. Seoane-Sepúlveda,
		\emph{Similarities and differences between real and complex Banach spaces:
			an overview and recent developments},
		Rev. R. Acad. Cienc. Exactas Fís. Nat. Ser. A Mat. RACSAM 116 (2022),
		article no. 88.
		
		\bibitem{MunozSarantopoulosTonge}
		G. Muñoz, Y. Sarantopoulos and A. Tonge,
		\emph{Complexifications of real Banach spaces, polynomials and multilinear maps},
		Studia Math. 134 (1999), 1--33.
		
		\bibitem{PappasRevesz}
		A. Pappas and S. G. Révész,
		\emph{Linear polarization constants of Hilbert spaces},
		J. Math. Anal. Appl. 300 (2004), 129--146.
		
		\bibitem{PinascoBombieri}
		D. Pinasco,
		\emph{Lower bounds for norms of products of polynomials via Bombieri inequality},
		Trans. Amer. Math. Soc. 364 (2012), 3993--4010.
		
		\bibitem{PinascoRn}
		D. Pinasco,
		\emph{On the \(n\)-th linear polarization constant of \(\mathbb R^n\)},
		Math. Nachr. 296 (2023), no. 8, 3593--3605.
		
		\bibitem{ReveszSarantopoulos}
		S. G. Révész and Y. Sarantopoulos,
		\emph{Plank problems, polarization and Chebyshev constants},
		J. Korean Math. Soc. 41 (2004), 157--174.
		
		\bibitem{Shannon}
		C. E. Shannon,
		\emph{A mathematical theory of communication},
		Bell System Technical Journal 27 (1948), no. 3, 379--423; no. 4, 623--656.
		
		
	\end{thebibliography}
\end{document}